\documentclass[11pt]{amsart}
\usepackage{amsopn}
\usepackage{amssymb, amscd}

\newcommand{\nc}{\newcommand}

\nc{\fg}{\mathfrak{f} } \nc{\vg}{\mathfrak{v} } \nc{\wg}{\mathfrak{w} }
\nc{\zg}{\mathfrak{z} } \nc{\ngo}{\mathfrak{n} } \nc{\kg}{\mathfrak{k} }
\nc{\mg}{\mathfrak{m} } \nc{\bg}{\mathfrak{b} } \nc{\ggo}{\mathfrak{g} }
\nc{\ggob}{\overline{\mathfrak{g}} } \nc{\sog}{\mathfrak{so} }
\nc{\sug}{\mathfrak{su} } \nc{\spg}{\mathfrak{sp} } \nc{\slg}{\mathfrak{sl} }
\nc{\glg}{\mathfrak{gl} } \nc{\cg}{\mathfrak{c} } \nc{\rg}{\mathfrak{r} }
\nc{\hg}{\mathfrak{h} } \nc{\tg}{\mathfrak{t} } \nc{\ug}{\mathfrak{u} }
\nc{\dg}{\mathfrak{d} } \nc{\ag}{\mathfrak{a} } \nc{\pg}{\mathfrak{p} }
\nc{\sg}{\mathfrak{s} } \nc{\affg}{\mathfrak{aff} }

\nc{\pca}{\mathcal{P}} \nc{\nca}{\mathcal{N}} \nc{\lca}{\mathcal{L}}
\nc{\oca}{\mathcal{O}} \nc{\mca}{\mathcal{M}} \nc{\tca}{\mathcal{T}}
\nc{\aca}{\mathcal{A}} \nc{\cca}{\mathcal{C}} \nc{\gca}{\mathcal{G}}
\nc{\sca}{\mathcal{S}} \nc{\hca}{\mathcal{H}} \nc{\bca}{\mathcal{B}}
\nc{\dca}{\mathcal{D}} \nc{\val}{\operatorname{val}}

\nc{\vp}{\varphi} \nc{\ddt}{\tfrac{{\rm d}}{{\rm d}t}} \nc{\dds}{\tfrac{{\rm d}}{{\rm d}s}}
\nc{\dpar}{\tfrac{\partial}{\partial t}} \nc{\im}{\mathtt{i}}

\nc{\SO}{\mathrm{SO}} \nc{\Spe}{\mathrm{Sp}} \nc{\Sl}{\mathrm{SL}}
\nc{\SU}{\mathrm{SU}} \nc{\Or}{\mathrm{O}} \nc{\U}{\mathrm{U}} \nc{\Gl}{\mathrm{GL}}
\nc{\Se}{\mathrm{S}} \nc{\Cl}{\mathrm{Cl}} \nc{\Spein}{\mathrm{Spin}}
\nc{\Pin}{\mathrm{Pin}} \nc{\G}{\mathrm{GL}_n(\RR)} \nc{\g}{\mathfrak{gl}_n(\RR)}

\nc{\RR}{{\Bbb R}} \nc{\HH}{{\Bbb H}} \nc{\CC}{{\Bbb C}} \nc{\ZZ}{{\Bbb Z}}
\nc{\FF}{{\Bbb F}} \nc{\NN}{{\Bbb N}} \nc{\QQ}{{\Bbb Q}} \nc{\PP}{{\Bbb P}}

\nc{\vs}{\vspace{.2cm}} \nc{\vsp}{\vspace{1cm}} \nc{\ip}{\langle\cdot,\cdot\rangle}
\nc{\ipp}{(\cdot,\cdot)} \nc{\la}{\left\langle} \nc{\ra}{\right\rangle} \nc{\unm}{\tfrac{1}{2}}
\nc{\unc}{\tfrac{1}{4}} \nc{\und}{\tfrac{1}{16}} \nc{\no}{\vs\noindent}
\nc{\lam}{\Lambda^2(\RR^n)^*\otimes\RR^n} \nc{\tangz}{{\rm T}^{\rm Zar}}
\nc{\nor}{{\sf n}}  \nc{\mum}{/\!\!/} \nc{\kir}{/\!\!/\!\!/}
\nc{\Ri}{\tfrac{4\Ric_{\mu}}{||\mu||^2}} \nc{\ds}{\displaystyle}
\nc{\ben}{\begin{enumerate}} \nc{\een}{\end{enumerate}} \nc{\f}{\frac}
\nc{\lb}{[\cdot,\cdot]} \nc{\isn}{\tfrac{1}{||v||^2}}
\nc{\gkp}{(\ggo=\kg\oplus\pg,\ip)} \nc{\ukh}{(\ug=\kg\oplus\hg,\ip)}
\nc{\tgkp}{(\tilde{\ggo}=\kg\oplus\pg,\ip)}
\nc{\wt}{\widetilde}
\nc{\raw}{\rightarrow} \nc{\lraw}{\longrightarrow}

\nc{\BF}[1]{\delta_{#1}\left(\Ricci_{#1}\right)}

\nc{\Hess}{\operatorname{Hess}} \nc{\ad}{\operatorname{ad}}
\nc{\Ad}{\operatorname{Ad}} \nc{\rank}{\operatorname{rank}}
\nc{\Irr}{\operatorname{Irr}} \nc{\End}{\operatorname{End}}
\nc{\Aut}{\operatorname{Aut}} \nc{\Inn}{\operatorname{Inn}}
\nc{\Der}{\operatorname{Der}} \nc{\Ker}{\operatorname{Ker}}
\nc{\Iso}{\operatorname{I}} \nc{\Diff}{\operatorname{Diff}}
\nc{\Lie}{\operatorname{L}} \nc{\tr}{\operatorname{tr}} \nc{\dif}{\operatorname{d}}
\nc{\sen}{\operatorname{sen}} \nc{\modu}{\operatorname{mod}}
\nc{\Ric}{\operatorname{R}} \nc{\Ricci}{\operatorname{Ric}} \nc{\Riem}{\operatorname{Rm}}
\nc{\vol}{\operatorname{vol}}
\nc{\sym}{\operatorname{sym}} \nc{\symac}{\operatorname{sym^{ac}}}
\nc{\symc}{\operatorname{sym^{c}}} \nc{\scalar}{\operatorname{sc}}
\nc{\grad}{\operatorname{grad}} \nc{\ricci}{\operatorname{Rc}}
\nc{\Nor}{\operatorname{Norm}} \nc{\riccic}{\operatorname{ric^{c}}}
\nc{\riccig}{\operatorname{ric^{\gamma}}} \nc{\Rin}{\operatorname{M}}
\nc{\Le}{\operatorname{L}} \nc{\tang}{\operatorname{T}}
\nc{\level}{\operatorname{level}} \nc{\rad}{\operatorname{r}}
\nc{\abel}{\operatorname{ab}} \nc{\CH}{\operatorname{CH}}
\nc{\mcc}{\operatorname{mcc}} \nc{\Adj}{\operatorname{Adj}}
\nc{\Order}{\operatorname{O}}

\theoremstyle{plain}
\newtheorem{theorem}{Theorem}[section]
\newtheorem{proposition}[theorem]{Proposition}
\newtheorem{corollary}[theorem]{Corollary}
\newtheorem{lemma}[theorem]{Lemma}

\theoremstyle{definition}

\theoremstyle{remark}
\newtheorem{remark}[theorem]{Remark}

\title{Scalar curvature behavior of homogeneous Ricci flows}

\author{Ramiro A. Lafuente}
\address{FaMAF and CIEM, Universidad Nacional de C\'ordoba, C\'ordoba, Argentina}
\email{rlafuente@famaf.unc.edu.ar}
\thanks{This research was partially supported by grants from CONICET (Argentina) and SeCyT (Universidad Nacional de C\'ordoba)}

\subjclass[2010]{53C44, 53C30}

\begin{document}

\maketitle

\begin{abstract}
We prove that the scalar curvature of a homogeneous Ricci flow solution blows up at a forward or backward finite-time singularity.
\end{abstract}


\section{Introduction}

A family of Riemannian metrics $(M^n,g(t))$ is a solution to the unnormalized Ricci flow starting at $g_0$ if it satisfies the evolution equation
\begin{align}
   \label{RF} \frac{\partial }{\partial t} g(t) &= -2 \ricci_{g(t)},\\
    \nonumber            g(0) &= g_0,
\end{align}
introduced by R. Hamilton in \cite{Ham82}. A natural question to ask for such a solution is what are the optimal geometric quantities that can control the formation of singularities, in the sense that if they remain bounded along an interval $[0,T)$ then the solution can be smoothly extended past time $T$ (see \cite{Knopf}).

 In the compact case, Hamilton proved in \cite{Ham95} that if the flow develops a singularity at a finite time $\omega$, then the norm of the Riemann curvature tensor $|\Riem(g(t))|$ is unbounded on $M\times [0,\omega)$. This was improved by N. Sesum in \cite{Ses}, who showed that $|\ricci (g(t))|$ must be unbounded on $M\times [0,\omega)$, where $\ricci (g(t))$ denotes the Ricci curvature. Furthermore, B. Wang proved in \cite{Wang} that a lower bound for the Ricci curvature together with an integral (rather than pointwise) bound for the scalar curvature of the form $\int_0^T \int_M |R|^\alpha \rm d \vol_{g(t)} \rm d t \leq C$, $\alpha\geq \tfrac{n+2}2$, allow one to extend the flow past time $T$.

It is expected that the scalar curvature $R(g(t))$ should also be unbounded in the presence of a finite-time singularity, however this question remains an open problem in the general case. In this direction, J. Enders, R. M\"uller and P. Topping proved in \cite{EMT} that if the singularity is of Type-I, then the scalar curvature must indeed blow up. It was shown by N. Le and N. Sesum in \cite{LeSes} that actually an integral scalar curvature bound is enough (in the compact case) in order to prevent Type-I singularities. On the other hand, in \cite{Zhang} Z. Zhang showed that the scalar curvature blows up at the first singular time of the K\"ahler Ricci flow.

The main purpose of this article is to prove the following theorem, which confirms the expected behavior for the scalar curvature in the case of homogeneous manifolds.
\begin{theorem}\label{main}
Let $(M,g(t))$ be a Ricci flow solution with maximal interval of definition $(\alpha,\omega)$, and assume that $(M,g(t))$ is homogeneous for all $t\in (\alpha,\omega)$. Let $R(g(t))$ denote the (constant on $M$) scalar curvature of $g(t)$.
\begin{itemize}
    \item[(i)] If $\omega<\infty$, then $R(g(t)) \underset{t\raw \omega}\lraw \infty.$
    \item[(ii)] If $\alpha > -\infty$, then $R(g(t)) \underset{t\raw \alpha}\lraw -\infty.$
\end{itemize}
\end{theorem}

Our main tool is an evolution equation for a curve of Lie brackets introduced in \cite{homRF} (called the \emph{bracket flow}), defined on the variety of Lie algebras of a certain fixed dimension, which turns out to be equivalent in a very precise sense to the Ricci flow of homogeneous manifolds. In fact, our proof of Theorem \ref{main} in Section \ref{scalar} relies only on the equivalence between homogeneous Ricci flow and the bracket flow (see Theorem \ref{homRF-BF}), and does not make use of any previous result from general Ricci flow theory.

We also study some properties of bracket flow solutions which are of independent interest, as well as some other links between the scalar curvature and the interval of definition of the flow. In addition, as an application of Theorem \ref{main}, we prove that any homogeneous Ricci flow solution on a manifold whose universal cover is diffeomorphic to $\RR^n$ is immortal (i.e. defined for all $t\in [0,\infty)$). This in particular applies to homogenenous Ricci flow solutions on solvmanifolds.

\vs \noindent {\it Acknowledgements.} I would like to thank my PhD advisor Jorge Lauret for his encouragement and support, and for his many helpful comments and suggestions.

\section{Preliminaries}

A Riemannian manifold $(M,g)$ is said to be \emph{homogeneous} if the isometry group $\Iso(M,g)$ acts transitively on $M$. If $M$ is connected (which we will always assume unless otherwise stated), each transitive, closed Lie subgroup $G\subseteq \Iso(M,g)$ gives rise to a presentation of $(M,g)$ as a homogeneous space with a $G$-invariant metric $(G/K, g)$ (a \emph{Riemannian homogeneous space}), where $K$ is the isotropy subgroup of $G$ at some point $o\in M$.

The presence of a $G$-invariant metric on $G/K$ implies the existence of a \emph{reductive} (i.e. $\Ad(K)$-invariant) decomposition $\ggo = \kg \oplus \pg$, where $\ggo$ and $\kg$ are respectively the Lie algebras of $G$ and $K$. Indeed, we can take $\pg$ as the orthogonal complement to $\kg$ with respect to the Killing form $B$ of $\ggo$ (see \cite[Lemma 2.1]{homRS}). Once such a decomposition has been chosen, we have the natural identification $\pg \equiv T_o M = T_{eK} G/K$, and any $G$-invariant metric on $G/K$ is determined by an $\Ad(K)$-invariant inner product $\ip$ on $\pg$.

In this situation, since we have assumed that $G\subseteq \Iso(M,g)$, we always have that the action of $G$ on $G/K$ is effective. However, in order to get a presentation of $(M,g)$ as a Riemannian homogeneous space $(G/K,g)$, it suffices to have an \emph{almost-effective} action (i.e. the normal subgroup $\{g\in G : ghK = hK, \forall h\in G \}$ is discrete). Moreover, it is easy to see that from a reductive decomposition of the Lie algebra $\ggo = \kg\oplus\pg$ and an $\Ad(K)$-invariant inner product $\ip$ on $\pg$, we can recover the Riemannian homogeneous space $(G/K,g)$, provided we choose $G$ simply connected (which we can always do without loosing almost-effectiveness).

\subsection{The space of homogeneous manifolds}\label{spacehmsection}
We will now describe a framework, developed in \cite{spacehm}, which is central in this paper because it allows us to work on the 'space of homogeneous manifolds'. It consists in parameterizing the set of $n$-dimensional homogeneous spaces with $q$-dimensional isotropy by a subset $\hca_{q,n}$ of the variety of $(q+n)$-dimensional Lie algebras.

Fix a real vector space $\ggo$ of dimension $q+n$ together with a direct sum decomposition $\ggo = \kg \oplus \pg$, where $\kg$ and $\pg$ are subspaces of dimension $q$ and $n$, respectively. Also, fix an inner product $\ip$ on $\pg$. Let $V_{q+n}$ be the space of all skew-symmetric algebra structures (or brackets) on $\ggo$, that is,
 \[
V_{q+n}:=\{\mu:\ggo\times\ggo\longrightarrow\ggo : \mu\; \mbox{bilinear and
skew-symmetric}\}.
\]
We can associate to a bracket $\mu\in V_{q+n}$ an $n$-dimensional homogeneous space with $q$-dimensional isotropy provided that the following conditions on $\mu$ are satisfied:
\begin{itemize}
    \item[(h1)] $\mu$ satisfies the Jacobi condition (i.e. it is a Lie bracket on $\ggo$), $\mu(\kg,\kg)\subset\kg$ and $\mu(\kg,\pg)\subset\pg$.

    \item[(h2)] If $G_\mu$ denotes the simply connected Lie group with Lie algebra $(\ggo,\mu)$ and $K_\mu$ is the connected Lie subgroup of $G_\mu$ with Lie algebra $\kg$, then $K_\mu$ is closed in $G_\mu$.

    \item[(h3)] $\ip$ is $\ad_{\mu}{\kg}$-invariant (i.e. $(\ad_{\mu}{Z}|_{\pg})^t=-\ad_{\mu}{Z}|_{\pg}$ for all $Z\in\kg$).

    \item[(h4)] $\{ Z\in\kg:\mu(Z,\pg)=0\}=0$.
\end{itemize}
(see \cite{spacehm} for further details). The subset $\hca_{q,n} \subseteq V_{q+n}$ is precisely the set of those $\mu$'s that satisfy the conditions, i.e.
\[
\hca_{q,n} = \{ \mu \in V_{q+n} : \mu \mbox{ satisfies the conditions } (\rm h1)-(\rm h4)\}.
\]

We can therefore associate to any $\mu \in \hca_{q,n}$ a Riemannian homogeneous space $(G_\mu/K_\mu, g_\mu)$, where $G_\mu$ is the simply connected Lie group with Lie algebra $(\ggo,\mu)$, $K_\mu$ is the connected Lie subgroup of $G_\mu$ with Lie algebra $(\kg, \mu|_{\kg\times\kg})$, which is closed by (h2), and $g_\mu$ is the $G_\mu$-invariant metric on $G_\mu/ K_\mu$ defined by the inner product $\ip$ on $\pg$ (i.e. $g_\mu (eK_\mu) = \ip$; recall that we identify $T_{eK_\mu} G_\mu/K_\mu$ with $\pg$).

If we extend the inner product $\ip$ on $\pg$ to an inner product on $\ggo$ (which for simplicity we will also denote by $\ip$), such that $\langle \kg,\pg\rangle = 0$, then we can endow $V_{q+n}$ with a natural inner product, defined by
\[
\langle \mu, \lambda \rangle = \sum_{i,j,k} \langle \mu(X_i,X_j),X_k \rangle \langle \lambda(X_i,X_j),X_k\rangle,
\]
where $\{ X_i \}$ is any orthonormal basis of $(\ggo,\ip)$. Thus we can refer to the norm of a bracket $\mu$ with respect to this inner product, which we will denote by $|\mu|$. This will be a very important feature in our study of homogeneous Ricci flow.

\subsection{The Ricci flow of homogeneous manifolds and the bracket flow}\label{homRF-BF}
We present here our main tool, introduced in \cite{homRF}, which is also the object of study of this article: the bracket flow. It is an ODE defined in the set $\hca_{q,n}$, which turns out to be equivalent in a very precise sense to the Ricci flow of homogeneous manifolds. We refer the reader to \cite[Section 3]{homRF} for a more detailed treatment.

Let $(M,g(t))$ be a Ricci flow solution as in \eqref{RF}. If the initial metric $g(0)=g_0$ is homogeneous, it can be easily proved that there exists a solution to \eqref{RF} which consists of homogeneous metrics, because the homogeneity allows us to reduce the problem to an ODE for a curve of inner products on the tangent space at some point $p\in M$. The general Ricci flow theory therefore implies that this solution $g(t)$ is unique among complete, bounded curvature metrics (see \cite{Shi}, \cite{ChenZhu}). Moreover, in the homogeneous case it is not hard to see that one must have $\Iso(M,g(t)) = \Iso(M,g_0)$ for all $t$ (this was proved in \cite{Kotsch} in the general case).

To sum up, for a homogenous metric $g_0$ on $M$ one always has a homogenous solution $g(t)$ to \eqref{RF} such that $g(0) = g_0$, defined on some maximal interval of time $(\alpha,\beta)$ with $0\in(\alpha,\beta)$, and the isometry group is preserved along the flow.

The preservation of the isometry group implies that a $G$-invariant metric remains $G$-invariant along the Ricci flow, thus we can study its evolution by restricting ourselves to $G$-invariant metrics on a fixed homogeneous space $G/K$, and this takes us directly into the space $\hca_{q,n}$. In order to understand how the Ricci flow looks in $\hca_{q,n}$, we consider an evolution equation defined in $V_{q+n}$. We say that a smooth curve $\mu(t)$ in $V_{q+n}$ is a \emph{bracket flow} solution if it satisfies
\begin{equation}\label{BF}
\frac{{\rm d}}{{\rm d}t} \mu = -\pi \left( \left[\begin{smallmatrix} 0 & 0\\ 0 & \Ricci_\mu \end{smallmatrix} \right]\right) \mu,
\end{equation}
where the blocks are according to the decomposition $\ggo = \kg \oplus \pg$, $\pi : \glg_{q+n}(\RR) \lraw \End(V_{q+n})$ is the usual representation given by
\[
\pi(A)\mu = A\mu(\cdot,\cdot) - \mu(A\cdot,\cdot) - \mu(\cdot,A\cdot),
\]
and $\Ricci_\mu\in \End(\pg)$ is defined as in \cite[\S 2.3]{homRF} and coincides with the Ricci operator of the Riemannian homogeneous space $(G_\mu/K_\mu, g_\mu)$ when $\mu\in \hca_{q,n}$.

It is easy to see that if $\mu_0 = \mu(0) \in \hca_{q,n}$, then $\mu(t) \in \hca_{q,n}$ for all $t$. Thus by considering solutions to \eqref{RF} and \eqref{BF} we have a priori two families of Riemannian homogeneous spaces
\[
(G/K,g(t)), \qquad (G_{\mu(t)}/K_{\mu(t)}, g_{\mu(t)}).
\]
The next result shows that bracket flow solutions differ from Ricci flow solutions only by pull-back by time-dependent diffeomorphisms.
\begin{theorem}\cite{homRF}\label{RF-BF}
Let $\mu_0 \in \hca_{q,n}$, and consider $(G_{\mu_0}/K_{\mu_0}, g_{\mu_0})$ the corresponding Riemannian homogeneous space. Let $g(t), \mu(t)$ be the solutions to \eqref{RF}, \eqref{BF} with values at $t=0$ given by $g_{\mu_0}$, $\mu_0$, respectively. Then $g(t)$ and $\mu(t)$ have the same maximal interval of definition, and there exist time-dependent, equivariant diffeomorphisms $\varphi(t) : G_{\mu_0}/K_{\mu_0} \lraw G_{\mu(t)}/K_{\mu(t)}$ such that
\[
    g(t) = \varphi(t)^*g_{\mu(t)}.
\]
\end{theorem}

This theorem allows us to address questions about the geometry and interval of definition of homogeneous Ricci flow by means of the bracket flow.

In all of what follows, unless otherwise stated, whenever we speak of an interval of definition for a flow (either Ricci or bracket) of the form $(\alpha,\omega)$, we will always mean its \emph{maximal} interval of definition. Moreover, if we say that a solution develops a singularity in finite time $\omega>0$, we mean that the solution to \eqref{RF} or \eqref{BF} has an interval of definition of the form $(\alpha,\omega)$, with $-\infty \leq \alpha < 0$ and $\omega < \infty$ (and analogously for a 'backward' singularity in finite time $\alpha<0$). Finally, a solution with interval of existence of the form $(\alpha,\infty)$ is called \emph{immortal}, one with $(-\infty,\omega)$ is called \emph{ancient}, and one with $(-\infty,\infty)$ is called \emph{eternal}, as usual.

\section{The bracket norm along a bracket flow solution}

In this section we will discuss the behavior of the bracket norm of a bracket flow solution having a finite (forward or backward) time of extinction. Our setting here is the one explained in Sections \ref{spacehmsection} and \ref{homRF-BF}. Recall that, since we have extended the inner product to $\ggo$, we have a natural inner product on $V_{q+n}$, and hence on $\hca_{q,n}$.

We begin with a simple but useful estimate.

\begin{lemma}\label{estimate1} If $\mu(t)$ is a bracket flow solution, then,
\[
    \left|\ddt \mu \right| \leq C |\Ricci_\mu | |\mu| \leq  \tilde{C} |\mu|^3,
\]
where $C, \tilde{C}$ are constants that only depend on $n$.
\end{lemma}

\begin{proof}
We see that $\pi(A)\mu$ is linear in $\mu$, thus the first inequality follows from the definition of a bracket flow solution. On the other hand, by using that each entry of $\Ricci_\mu$ (or, more precisely, of its corresponding matrix with respect to an orthonormal basis) is a homogeneous quadratic polynomial, we get
\[
|\Ricci_\mu | \leq C_1 |\mu|^2,
\]
and the second inequality follows.
\end{proof}

From standard ODE theory, it is clear that in the presence of a finite-time singularity at time $\omega<\infty$, there exists a sequence $t_k \raw \omega$ such that $|\mu(t_k)| \raw \infty$. The next result improves this fact by showing that, actually, $|\mu(t)|\underset{t\raw\omega}\lraw \infty$. An analogous result also holds in the case of backward, finite-time singularities.

\begin{proposition}\label{cotainfmu}
Let $\mu(t)$ be a bracket flow solution with maximal interval of definition $(\alpha, \omega)$. Then there exists a constant $C = C(n)>0$, such that:
\begin{itemize}
    \item[(i)] If $\omega<\infty$, then
        \[
        \left| \mu(t)\right| \geq \frac{C}{(\omega-t)^{1/2}}, \qquad \forall t\in [0,\omega).
        \]
    \item[(ii)] If $\alpha>-\infty$, then
        \[
        \left| \mu(t)\right| \geq \frac{C}{(t-\alpha)^{1/2}}, \qquad \forall t\in (\alpha,0].
        \]
\end{itemize}
\end{proposition}

\begin{proof}
We will focus on proving only part (i), since part (ii) follows analogously by reversing the sign of the time variable.

Let $t_0 \in (\alpha,\omega)$. From Lemma \ref{estimate1} we have that
\[
\ddt |\mu |^2 = 2\la \mu, \ddt \mu \ra \leq C |\mu |^4, \qquad t\in [t_0,\omega).
\]
(actually we may take the same constant $C$ for all $t_0\in [0,\omega)$). So by comparison we obtain that $|\mu(t)| \leq f(t)$,  for all $t\in [t_0,\omega)$, where $f$ is given by
\[
 f' = C f^2, \quad f(t_0) = | \mu(t_0) |^2.
\]
Thus,
\[
| \mu(t) |^2 \leq \frac1{-C(t-t_0) + |\mu(t_0)|^{-2}},
\]
and then $|\mu(t)|$ cannot blow up before the time $t = t_0 + \frac1C \left| \mu(t_0)\right|^{-2}$. This implies that $\omega \geq t_0 + \frac1C \left| \mu(t_0)\right|^{-2}$, and the result follows.
\end{proof}

\begin{remark}
If we were able to obtain also an estimate of the form $|\mu(t)| \leq \tfrac{C}{(\omega-t)^{1/2}}$, which is the case in all known examples, then it can be shown that the corresponding Ricci flow solution $g(t)$ develops a Type-I singularity. Indeed, in that case we can argue as in \cite[Section 6.2]{nilRF} to obtain
\[
   | \Riem(g(t))| =  | \Riem(\mu(t))| \leq C_n |\mu(t)|^2  \leq \frac{C}{\omega-t}.
\]
\end{remark}

\section{Scalar curvature and the interval of definition of the flow}\label{scalar}

The aim of this section is to deal with the interplay between the scalar curvature and the maximal interval of definition of a homogeneous Ricci flow solution, by proving Theorem \ref{main} and the following proposition.

\begin{proposition}
Let $(M,g(t))$, $g(0) = g_0$, be a Ricci flow solution with maximal interval of definition $(\alpha, \omega)$, $\alpha<0<\omega$, such that for each $t\in(\alpha,\omega)$ the scalar curvature $R(g(t))$ is constant on $M$ (which holds in particular if $g(t)$ is homogeneous).
\begin{enumerate}\label{signoR}
  \item[(i)] If $R(g_0)>0$, then $\omega \leq \frac{n}2 R(g_0)^{-1}$.
  \item[(ii)] If $R(g_0)<0$, then $\alpha \geq \frac{n}2 R(g_0)^{-1}$.
\end{enumerate}
\end{proposition}

\begin{proof}
The evolution of the scalar curvature along a Ricci flow solution implies that
\[
\frac{\partial}{\partial t} R \geq \Delta R + \frac{2}{n} R^2
\]
(see for instance \cite[Corollary 2.5.5]{Topp}). Our assumption on the scalar curvature implies that $\Delta R \equiv 0$, thus $R$ satisfies
\begin{equation}\label{desigR}
\frac{\rm d}{\rm d t} R \geq \frac2{n} R^2,
\end{equation}
and by comparison we see that
\[
    R(g(t)) \geq (-\tfrac{2}{n} t + R(g_0)^{-1})^{-1}.
\]
Now assume that $R(g_0) > 0$. The denominator $-\tfrac{2}{n}t + R(g_0)^{-1}$ cannot vanish for $t\in [0,\omega)$, and since it is positive at $t=0$, we have that
\[
-\tfrac{2}{n}t + R(g_0)^{-1} > 0, \qquad \forall t\in [0,\omega),
\]
which, by taking the limit as $t\rightarrow \omega$, implies part (i).

The proof of part (ii) is completely analogous, by reversing the sign of the variable $t$.
\end{proof}

\begin{remark}
In the homogeneous case, one can obtain inequality \eqref{desigR} by only using the evolution equation of the scalar curvature along the bracket flow:
\begin{equation}\label{scevol}
\frac{\rm d}{\rm d t} R(\mu(t)) = 2 \tr \Ricci_{\mu(t)}^2
\end{equation}
(see \cite[Proposition 3.8, (vi)]{homRF}).
\end{remark}

This implies that an immortal homogeneous Ricci flow solution must have $R\leq0$ for all $t$, and that an ancient homogeneous Ricci flow solution must have $R\geq0$ for all $t$. The inequalities are both strict unless $R\equiv 0$, in which case $g(t)$ is flat for all $t$. We obtain in particular that there are no non-flat, eternal homogeneous solutions to the Ricci flow.

We next prove Theorem \ref{main}, which is the main result of the present paper.


\begin{proof}[Proof of Theorem \ref{main}]
Let us prove part (i). By using Theorem \ref{RF-BF}, we see that it suffices to show that if $\mu(t)$ is a bracket flow solution with a finite-time singularity at $\omega<\infty$, then  $R(\mu(t)) \underset{t\raw \omega} \lraw \infty$. Indeed, that result shows that the Ricci flow of homogeneous manifolds is equivalent, up to pull back by diffeomorphisms, to the bracket flow, and their intervals of definition coincide.

 To prove that, observe that from Lemma \ref{estimate1} we have that any bracket flow solution must satisfy
\[
  \ddt |\mu|^2 = 2\la \ddt \mu, \mu \ra \leq 2 \left| \ddt \mu \right| \left|\mu\right|   \leq C | \mu |^2 |\Ricci_\mu |,
\]
so by integrating on $[0,t)$, for any $t\in [0,\omega)$, we obtain
\[
 \log |\mu(t)|^2 - \log |\mu(0)|^2  = \int_0^t \dds \log |\mu(s)|^2 \rm d s \leq C \int_0^t |\Ricci_{\mu(s)}| \rm d s.
\]
Thus, $\int_0^\omega |\Ricci_{\mu(s)}| \rm d s = \infty$, since $\mu(t)$ is unbounded. On the other hand, for $t\in [0,\omega)$, we use \eqref{scevol} to get
\begin{align*}
    \int_0^t |\Ricci_{\mu(s)}| \rm d s  &\leq  \int_0^t \unm\left(1 + |\Ricci_{\mu(s)}|^2 \right) \rm d s  \\
        &= \tfrac{t}{2} + \tfrac14 \int_0^t \dds R(\mu(s)) \rm d s \\
        &= \tfrac{t}2 + \tfrac14 \left( R(\mu(t)) - R(\mu(0)) \right),
\end{align*}
and part (i) follows by letting $t \raw \omega$.

The proof of part (ii) is completely analogous.
\end{proof}

As an application of Theorem \ref{main}, we obtain that homogeneous Ricci flow solutions on $\RR^n$ (and in particular on solvmanifolds, since they are all diffeomorphic to a quotient of $\RR^n$) are always immortal.

\begin{corollary}\label{solvmanifolds}
Let $M = G/K$ be a homogeneous space. If the universal cover of $M$ is diffeomorphic to $\RR^n$, then the Ricci flow solution $g(t)$ starting at any $G$-invariant metric $g_0$ on $M$ exists for all $t\geq 0$. On the other hand, if the universal cover of $M$ is not diffeomorphic to $\RR^n$, then there exists a $G$-invariant metric $g_0$ on $M$ such that the Ricci flow $g(t)$ starting at $g_0$ develops a singularity in finite time.
\end{corollary}

\begin{proof}
By \cite[Th\'eoreme 2]{BB}, the universal cover of $M$ is diffeomorphic to $\RR^n$ if and only if any $G$-invariant metric on $M$ has non-positive scalar curvature. Thus if this is the case, $R(g(t))\leq0$ for all $t$, and Theorem \ref{main} yields that we cannot have a finite-time singularity. The last assertion follows from Lemma \ref{signoR}, (i), since on such a manifold there exist $G$-invariant metrics with positive scalar curvature.
\end{proof}

\begin{remark}
One could state a result analogous to Corollary \ref{solvmanifolds} but for the case of positive scalar curvature and backward time of existence. However, the conclusion in this case is trivial, because of the rigidity imposed by the positivity of the scalar curvatures. Indeed, recall that to ensure that every $G$-invariant metric on a homogeneous space $G/K$ has non-negative scalar curvature, we must have that every $G$-invariant metric is (locally) the Riemannian product of a flat metric and an invariant metric on a compact homogeneous space of \emph{normal type} (in the sense of B\'erard-Bergery, see \cite{BB}). Furthermore, such a metric is isometric to the Riemannian product of a flat metric and invariant metrics on compact, isotropy-irreducible homogeneous spaces (see \cite{Ber}), which are necessarily Einstein. Thus a Ricci flow solution starting at any of those metrics is given by the product of Ricci flow solutions on each factor, and this is clearly an ancient solution.
\end{remark}

\end{document}